\documentclass[a4paper]{amsart}
\usepackage{hyperref}

\newtheorem{thm}[equation]{Theorem}

\newtheorem{prop}[equation]{Proposition}
\newtheorem{lemma}[equation]{Lemma}

\theoremstyle{remark}
\newtheorem{rema}[equation]{Remark}

\numberwithin{equation}{section}

\newcommand{\settc}[2]{\{\,#1 \bigm\vert #2\,\}}

\newcommand{\abs}[1]{\left\vert #1 \right\vert}

\newcommand{\norm}[1]{\lVert #1 \rVert}
\newcommand{\norma}[1]{\lvert #1 \rvert}

\newcommand{\tra}[1]{\gamma(#1)}

\title[pseudo-relativistic Hartree equations of critical type]{Ground
states for pseudo-relativistic Hartree equations of critical type}

\author{Vittorio Coti Zelati}
\email[Coti Zelati]{zelati@unina.it}
\address[Coti Zelati]{Dipartimento di Matematica Pura e Applicata 
``R.~Caccioppoli''\\
Universit\`a di Napoli ``Federico II''\\
via Cintia, M.S.~Angelo\\
80126 Napoli (NA), Italy}

\author{Margherita Nolasco}
\email[Nolasco]{nolasco@univaq.it}
\address[Nolasco]{Dipartimento di Matematica Pura e Applicata,
Universit\`a dell'Aquila
via Vetoio, Loc. Coppito	
67010 L'Aquila AQ
Italia}

\thanks{Work partially supported by the PRIN2009 grant ``Critical
Point Theory and Perturbative Methods for Nonlinear Differential
Equations''.}

\begin{document}

\begin{abstract}
We study the existence of ground state solutions for a class of non-
linear pseudo-relativistic Schr\"odinger equations with critical
two-body interactions.  Such equations are characterized by a nonlocal
pseudo-differential operator closely related to the square-root of the
Laplacian.  We investigate such a problem using variational methods
after transforming the problem to an elliptic equation with a
nonlinear Neumann boundary conditions.
\end{abstract}

\maketitle

\section{Introduction}

The relativistic Hamiltonian for $N$ identical particles of mass $m$,
position $x_{i}$ and momentum $p_{i}$ interacting through the two body
potential $\alpha W(\abs{x_{i} - x_{j}})$ is given by
\begin{equation*}
    \mathcal{H} = \sum_{i=1}^{N} \bigl(\sqrt{p_{i}^{2}c^{2} +
    m^{2}c^{4}} - mc^{2}\bigr) - \alpha \sum_{i \neq j} W(\abs{x_{i} -
    x_{j}}).
\end{equation*}
where $c$ is the speed of light and $\alpha >0$ is a coupling
constant.

According to the usual quantization rules the dynamics of the
corresponding system of $N$-identical quantum spinless particles (a
\emph{Bose gas}) is described by the complex wave function
$\Psi_N=\Psi_N(t, x_1, \dots, x_N)$ governed by the Schr\"odinger
equation
\begin{equation*}
    i \hbar \partial_{t} \Psi_N = \mathcal{H}_{N} \Psi_N
\end{equation*}
where $\hbar $ is the Planck's constant.  Here $\mathcal{H}_{N} \colon
\mathcal{D} \subset L^2(\mathbb{R}^3)^{\otimes_{s} N } \to
L^2(\mathbb{R}^3)^{\otimes_{s} N}$ is the \emph{quantum mechanics}
Hamiltonian operator, obtained from the classical Hamiltonian with the
usual quantization rule $p \mapsto -i\hbar\nabla$, and defined in a
suitable dense domain $\mathcal{D}$.  In the case we are interested in
$\mathcal{H}_{N}$ is
\begin{equation*}
    \mathcal{H}_{N} = (\sum_{j=1}^{N} \sqrt{-\hbar^{2} c^{2} \Delta_j
    + m^{2}c^{4}} - mc^2) - \alpha \sum_{i \not = j}^N W(|x_i - x_j|).
\end{equation*}
where $W$ is the multiplication operator corresponding to the two body
interaction potential, (e.g. $W(|x|) = |x|^{-1}$ for gravitational
interactions).

The operator (from now on we will take $\hbar = 1$, $c = 1$)
\begin{equation}
    \label{eq:radicelaplace}
    \sqrt{-\Delta + m^{2}}
\end{equation}
can be defined for all $f \in H^{1}(\mathbb{R}^{N})$ as the inverse
Fourier transform of the $L^{2}$ function $ \sqrt{\abs{k}^{2} + m^{2}}
\mathcal{F}[f] (k) $ (here $\mathcal{F}[f]$ denotes the Fourier
transform of $f$) and it is also associated to the quadratic form
\begin{equation*}
    \mathcal{Q}(f,g) = \int_{\mathbb{R}^{N}} \sqrt{ \abs{k}^{2} +
    m^{2}} \, \mathcal{F}[f] \mathcal{F}[g] \, dk
\end{equation*}
which can be extended to the space
\begin{equation*}
    H^{1/2}(\mathbb{R}^{N}) = \settc{ f \in L^{2}(\mathbb{R}^{N}) }{
    \int_{\mathbb{R}^{N}} |k| |\mathcal{F}[f](k)|^{2} \, dk < +
    \infty}
\end{equation*}
(see e.g. \cite{LiebLoss} for more details).

In the mean field limit approximation (i.e. $\alpha N \simeq O(1)$ as
$ N \to + \infty$) of a quantum relativistic Bose gas, one is lead to
study the nonlinear mean field equation --- called \emph{the
pseudo-relativistic Hartree equation} --- given by
\begin{equation}
    \label{eq:pseudorel_NLS}
    i \partial_{t}\psi = (\sqrt{- \Delta + m^{2}} - m ) \psi - ( W
    \ast |\psi|^{2}) \psi .
\end{equation}
where $\ast$ denotes convolution.  We will take attractive two body
interaction, and hence $W$ will always   be a nonnegative  function.  

See \cite{LiebYau87} for the study of this equation when $W$ is the
gravitational interaction, and \cite{ElgartSchlein07} for a rigorous
derivation of the mean field equation \eqref{eq:pseudorel_NLS} as a
limit as $N \to +\infty$ of the Schr\"odinger equation for $N$ quantum
particles, and \cite{DallAcquaSorensenOstergaardStockmeyer2008} for
more recent development on models involving the pseudorelativistic
operator $\sqrt{-\Delta + m^{2}}$.

It has recently been proved that for Newton or Yukawa-type two body
interactions (i.e. $W(|x|) = \abs{x}^{-1}$ or $\abs{x}^{-1}
e^{-\abs{x}}$ in $\mathbb{R}^{3}$) such an equation is locally
well-posed in $H^{s}$, $s \geq 1/2$, and that the solution is global
in time for small initial data in $L^{2}$ (see \cite{Lenzmann07}).
Blow up has been proved in \cite{FrohlichLenzmann07-1,
FrohlichLenzmann07}.

Due to the \emph{focusing} nature of the nonlinearity (attractive
two-body interaction) there exist \emph{solitary waves} solutions
given by
\begin{equation*}
    \psi(t,x) = \text{e}^{i \mu t} \varphi(x)
\end{equation*}
where $\varphi $ satisfy the nonlinear eigenvalue equation
\begin{equation}
    \label{eq:stationary}
    \sqrt{- \Delta + m^{2}} \varphi - m \varphi- ( W \ast |
    \varphi|^2) \varphi = - \mu \varphi .
\end{equation}
In \cite{LiebYau87} the existence of such solutions (in the case $W(x)
= \abs{x}^{-1}$) has been proved provided that $M < M_{c}$, $M_{c}$
being the \emph{Chandrasekhar limit mass}.  More precisely the Authors
have shown the existence in $H^{1/2}(\mathbb{R}^3)$ of a radial,
real-valued non negative minimizer (\emph{ground state}) of
\begin{equation}
    \label{eq:static_energy}
    \mathcal{E}[\psi] = \frac{1}{2} \int_{\mathbb{R}^{3}} \bar \psi (
    \sqrt{- \Delta + m^{2}} -m)\psi \, dx - \frac{1}{4}
    \int_{\mathbb{R}^{3}} ( |x|^{-1} \ast |\psi|^2) |\psi|^2 \, dx.
\end{equation}
with given fixed ``mass-charge'' $ M = \int_{\mathbb{R}^{3}} |\psi|^2
\, dx < M_{c}$.  We call mass-critical the potentials $W$ whose
associated functional $\mathcal{E}$ exhibits this kind of phenomenon.

More recently in \cite{FrohlichJonssonLenzmann07} it has been proved
that the ground state solution is regular ($H^{s}(\mathbb{R}^{3})$,
for all $s \geq 1/2$), strictly positive and that it decays
exponentially.  Moreover the solution is unique, at least for small
$L^{2}$ norm (\cite{Lenzmann09}).

Let us remark that these last results are heavily based on the
specific form (Newton or Yukawa type) of the two body
interactions in the Hartree nonlinearity.  Indeed in these cases the
estimates of the nonlinearity relies on the following facts
\begin{itemize}
    \item for this class of potentials one has that
    \begin{equation*}
	\frac{e^{-\mu\abs{x}}}{4\pi\abs{x}} * f = (\mu^{2} - \Delta)^{-1}f
	\quad \text{ for } f \in \mathcal{S}(\mathbb{R}^{3}), \ \mu \geq 0
    \end{equation*}

    \item the use of   a generalized Leibnitz rule for Riesz and
    Bessel potentials

    \item the following estimate holds
    \begin{equation*}
	\|\frac{1}{ |x|} \ast |u|^2\|_{L^{\infty}} \leq \frac{\pi}{2}
	\| (-\Delta)^{1/4} u \|^2_{L^{2}}.
    \end{equation*}
\end{itemize}

In \cite{CZNolasco2011} it has been proved an existence and regularity
result for the solutions of \eqref{eq:stationary} for a wider class of
nonlinearities
by exploiting the relation of equation
\eqref{eq:stationary} with an elliptic equation on
$\mathbb{R}^{N+1}_{+}$ with a nonlinear Neumann boundary condition.
Such a relation has been recently used to study several problems
involving fractional powers of the laplacian (see
e.g.~\cite{CabreMorales05} and references therein) and it is based on
an alternative definition of the operator \eqref{eq:radicelaplace}
that can be described as follows.  Given any function $u \in
\mathcal{S}(\mathbb{R}^{N})$ there is a unique function $v \in
\mathcal{S}(\mathbb{R}^{N+1}_{+})$ (here $\mathbb{R}^{N+1}_{+} =
\settc{(x,y) \in \mathbb{R} \times \mathbb{R}^{N}}{x > 0}$) such that
\begin{equation*}
    \begin{cases}
	-\Delta v + m^{2} v = 0 & \text{in } \mathbb{R}^{N+1}_{+} \\
	v(0,y) = u(y) & \text{for } y \in \mathbb{R}^{N} = \partial
	\mathbb{R}^{N+1}_{+}. \end{cases}
\end{equation*}
Setting 
\begin{equation*}
    Tu(y) = - \frac{\partial v}{\partial x}(0,y)
\end{equation*}
we have that the equation
\begin{equation*}
    \begin{cases}
	-\Delta w + m^{2}w =0 & \text{in } \mathbb{R}^{N+1}_{+} \\
	w(0,y) = Tu(y) = - \frac{\partial v}{\partial x}(0,y) &
	\text{for } y \in \mathbb{R}^{N}
    \end{cases}
\end{equation*}
has the solution $w(x,y) = -\frac{\partial v}{\partial x}(x,y)$.  From
this we have that
\begin{equation*}
    T(Tu)(y) = - \frac{\partial w}{\partial x}(0,y) =
    \frac{\partial^{2} v}{\partial x^{2}} (0,y) =( -\Delta_{y} v +
    m^{2} v)(0,y)
\end{equation*}
and hence $T^{2} = (-\Delta_{y} + m^{2})$.

In \cite{CZNolasco2011} we have studied the equation
\begin{equation}
    \sqrt{-\Delta + m^{2}} v = \mu v + \nu \abs{v}^{p-2}v + \sigma
    (W*\abs{v}^{2}) v \qquad \text{in } \mathbb{R}^{N}
    \label{eq:stazionaria}
\end{equation}
where $p\in (2, \frac{2N}{N-1})$, $\mu < m$ is fixed , $\nu$, $\sigma
\geq 0$ (but not equal $0$ both), $W \in L^{r}(\mathbb{R}^{N}) +
L^{\infty}(\mathbb{R}^{N})$, $W \geq 0$, $(r > N/2)$, $W(x) =
W(\abs{x}) \to 0$ as $\abs{x} \to+\infty$.

The results are obtained, following the approach outlined above, by
studying the following equivalent elliptic problem with nonlinear
boundary condition
\begin{equation}
\label{eq:Neumann_halfplane}
    \begin{cases}
	-\Delta v + m^{2}v =0 & \text{in } \mathbb{R}^{N+1}_{+} \\
	- \frac{\partial v}{\partial x} = \mu v + \nu \abs{v}^{p-2}v +
	\sigma (W*\abs{v}^{2}) v & \text{on } \mathbb{R}^{N} =
	\partial\mathbb{R}^{N+1}_{+}
    \end{cases}
\end{equation}
and the associated functional on $H^{1}(\mathbb{R}^{N+1}_{+})$.

Let us point out that in \cite{CZNolasco2011} the $L^{2}$ norm of the
solution is not prescribed.  In such a case existence of a (positive,
radially symmetric) solution can be proved for a certain class of
potentials $W$ and exponents $p$ which is larger then the one we can
deal with here.

When the $L^{2}$ norm is prescribed to be $M$ (the most relevant
problem from a physical point of view), as in \cite{LiebYau87}, then
the Newton potential ($\abs{x}^{-1}$ in $\mathbb{R}^{3}$) is critical,
in the sense that minimization of $\mathcal{E}$ given by
\eqref{eq:static_energy} is possible only when $M < M_{c}$
(see theorem \ref{thm:Theo}).

The main purpose of this paper is to exploit this approach also for
the problem of finding minimizer of the static energy
\begin{equation} 
    \label{eq:static-nostra}
    \mathcal{E}[u] = \frac{1}{2} \int_{\mathbb{R}^{N}} u ( \sqrt{-
    \Delta + m^{2}} -m)u \, dx + \frac{\eta}{p} \int_{\mathbb{R}^{N}}
    |u|^{p} \, dx - \frac{\sigma}{4} \int_{\mathbb{R}^{N}} (W \ast
    |u|^2) |u|^2 \, dx.  
\end{equation}
with prescribed $L^2$ norm, for a wider class of attractive two-body
potential including the critical case.

To be more precise, we consider a class of two body potential $W \in
L^q_{w}(\mathbb{R}^N)$, with $q \geq N$.  We recall that
$L^q_{w}(\mathbb{R}^N)$, the weak $L^{q}$ space, is the space of all
measurable functions $f$ such that
\begin{equation*}
    \sup_{\alpha >0 } \alpha | \settc{ x }{|f(x)| > \alpha}|^{1/q} < +
    \infty,
\end{equation*}
where $|E|$ denotes the Lebesgue measure of a set $E \subset
\mathbb{R}^N $.  Note that $W(x) = |x|^{-1}$ does not belong to any
$L^q$-space but it belongs to $ L^N_{w}(\mathbb{R}^N)$.  We say that a
potential $W$ is \emph{critical} if $W \in L^{N}(\mathbb{R}^{N})$.

Our main result is the following

\begin{thm}
\label{thm:Theo}
    Let $W \in L^{q}_{w}(\mathbb{R}^{N})$, $q \geq N \geq 2$, $W(y)
    \geq 0$ for all $y \in \mathbb{R}^{N}$ and such that
    \begin{equation}
	\label{eq:omogeneitaW}
	W(\lambda^{-1}y) \geq \lambda^{\alpha} W(y), \qquad \text{for
	all } \lambda \in (0,1) \text{ and for some } \alpha > 0.
    \end{equation}
    We also assume that $W(x) = W(\abs{x})$ is rotationally symmetric
    and that $W(r) \to 0$ as $r \to + \infty$.
    
    Take $\eta \geq 0$, $\sigma > 0$ and $p \in (2 + \frac{2}{q}, 2 +
    \frac{2}{N-1} = \frac{2N}{N-1}]$.  Then 
    \begin{itemize}
	\item if $\eta > 0$ or $\eta = 0$ and $q > N$, for all $M > 0$
	there is a strictly positive minimizer $u \in
	H^{1/2}(\mathbb{R}^{N})$ of $\mathcal{E}[u]$ such that
	$\int_{\mathbb{R}^{N}}u^{2} = M$.
    
	\item (\emph{mass-critical case}) if $\eta = 0$ and $q = N$
	there is a critical value $M_{c} > 0$ such that for all $0 < M
	< M_{c}$ there is a strictly positive minimizer $u \in
	H^{1/2}(\mathbb{R}^{N})$ of $\mathcal{E}[u]$ such that
	$\int_{\mathbb{R}^{N}}u^{2} = M$.
    \end{itemize}
    Moreover there exists $\mu >0$ such that $u$ is a smooth,
    exponentially decaying at infinity, solution of
    \begin{equation*}
	(\sqrt{-\Delta + m^{2}} - m)u = -\mu u - \eta \abs{u}^{p-2}u +
	\sigma (W * \abs{u}^{2})u \qquad \text{in} \ \mathbb{R}^{N}.
    \end{equation*}
    and $u$ is radial if $W=W(r)$ is a decreasing function of $r > 0$.
\end{thm}

\begin{rema}
    The nonlinear term $\abs{u}^{p-2}u$ is a defocusing nonlinearity,
    the convolution term is a focusing nonlinearity.  An open problem
    is to understand if solitons exists also for other ranges of $p$,
    in particular for $2 < p \leq 2 + \frac{2}{q}$ and $W \in
    L^{q}_{w}$.
\end{rema}
 
\begin{rema}  
    If $W \in L^{q}_{w}$ and \eqref{eq:omogeneitaW} holds for some
    $\alpha > 0$, then necessarily $\alpha \in (0, N/q]$.  If $W(x) =
    \abs{x}^{-\alpha}$, then $W \in L^{q}_{w}$ if and only if $\alpha
    = N/q$.
\end{rema}

\begin{rema}
    $\mu$ is a Lagrange multiplier.
\end{rema}

\section{Preliminaries }

Let $(x,y) \in \mathbb{R} \times \mathbb{R}^{N}$.  We have already
introduced $\mathbb{R}^{N+1}_{+} = \settc{(x,y) \in
\mathbb{R}^{N+1}}{x > 0}$.
With $\norm{u}_{p}$ we will always denote the norm of $u \in
L^{p}(\mathbb{R}^{N+1}_{+})$, with $\norm{u}$ the norm of $u \in
H^{1}(\mathbb{R}^{N+1}_{+})$ and with $\norma{v}_{p}$ the norm of $v
\in L^{p}(\mathbb{R}^{N})$.

We recall that for all $v \in H^{1}(\mathbb{R}^{N+1}) \cap
C^{\infty}_{0}(\mathbb{R}^{N+1})$
\begin{multline*}
    \label{eq:disProdottoa}
    \int_{\mathbb{R}^{N}} \abs{v(0,y)}^{p} dy = \int_{\mathbb{R}^{N}}
    dy \int_{+\infty}^{0} \frac{\partial \hfil}{\partial x}
    \abs{v(x,y)}^{p} dx \\
    \leq p \iint_{\mathbb{R}^{N+1}_{+}}
    \abs{v(x,y)}^{p-1}\abs{\frac{\partial v}{\partial x}(x,y)} \, dx
    \, dy \\
    \leq p \left(\iint_{\mathbb{R}^{N+1}_{+}} \abs{v(x,y)}^{2(p-1)} dx
    \, dy \right)^{1/2} \left(\iint_{\mathbb{R}^{N+1}_{+}}
    \abs{\frac{\partial v}{\partial x}(x,y)}^2 dx \, dy \right)^{1/2}
\end{multline*}
that is
\begin{equation}
    \norma{v(0,\cdot)}^{p}_{p} \leq p \norm{v}_{2(p-1)}^{p-1}
    \norm{\frac{\partial v}{\partial x}}_{2},
    \label{eq:disProdotto}
\end{equation}
which, by Sobolev embedding, is finite for all $2 \leq 2(p-1) \leq
2(N+1)/((N+1) - 2)$, that is $2 \leq p \leq 2^{\sharp}$, where we have
set $2^{\sharp} = 2N/(N-1)$.  By density of $H^{1}(\mathbb{R}^{N+1})
\cap C^{\infty}_{0}(\mathbb{R}^{N+1})$ in
$H^{1}(\mathbb{R}^{N+1}_{+})$ such an estimates allows us to define
the trace $\tra{v}$ of $v$ for all the functions $v \in
H^{1}(\mathbb{R}^{N+1}_{+})$.  The inequality
\begin{equation}
    \label{eq:ultimaLabel}
    \norma{\tra{v}}^{p}_{p} \leq p \norm{v}_{2(p-1)}^{p-1}
    \norm{\frac{\partial v}{\partial x}}_{2},
\end{equation}
holds then for all $v \in H^{1}(\mathbb{R}^{N+1}_{+})$.

It is known that the traces of functions in
$H^{1}(\mathbb{R}^{N+1}_{+})$ belong to $H^{1/2}(\mathbb{R}^{N})$ and
that every function in $H^{1/2}(\mathbb{R}^{N})$ is the trace of a
function in $H^{1}(\mathbb{R}^{N+1}_{+})$.  Then
\eqref{eq:ultimaLabel} is in fact equivalent to the well known fact
that $\tra{v} \in H^{1/2}(\mathbb{R}^{N}) \hookrightarrow
L^{q}(\mathbb{R}^{N})$ provided $q \in [2, 2^{\sharp}]$.  We also
recall here that
\begin{equation*}
    \norm{w}_{H^{1/2}}^{2} = \inf \settc{\norm{u}^{2}}{u \in
    H^{1}(\mathbb{R}^{N+1}_{+}), \ \tra{u} = w} =
    \int_{\mathbb{R}^{N}} (1 + \abs{\xi})\abs{\mathcal{F}w(\xi)}^{2}
    \, d\xi.
\end{equation*}

Let us also introduce the norm of the weak $L^q$-space as follows
\begin{equation*}
    \| f \|_{q,w} = \sup_{A} |A| ^{- 1/r} \int_{A} |f(x)| \, dx
\end{equation*}
where $1/q + 1/r = 1$ and $A$ denotes any measurable set of finite
measure $|A|$ (see e.g. \cite{LiebLoss} for more details).  Now using
this norm we can state the {\it weak Young inequality}.  If $g \in
L^q_w(\mathbb{R}^N)$, $f \in L^p(\mathbb{R}^N)$ and $h \in
L^r(\mathbb{R}^N)$ where $1 < q, p, r < + \infty$ and $1/q + 1/p + 1/
r = 2 $ then
\begin{equation}
    \label{eq:weakYang}
    \int_{\mathbb{R}^N}\int_{\mathbb{R}^N} f(y) g(y-z) h(y) \, dy \,
    dz \leq C_{p,q,r} \|g\|_{q, w} |f|_p |h|_r.
\end{equation}

We consider the class of two-body interactions $W \in
L^q_w(\mathbb{R}^{N})$ for $q \geq N$.  By weak Young inequality and
H\"older inequality we have for $r= 4q/(2q-1)$ ($\in (2,2^{\sharp})$
since $q \geq N$) and for all $p \in (4q/(2q-1), 2^{\sharp}]$
\begin{equation}
    \label{eq:stima_conv}
    \int_{\mathbb{R}^{N}} ( W* \abs{u}^2) \abs{w}^2 \, dy \leq C
    \norm{W}_{q, w} \abs{w}^{4}_r \leq C \norm{W}_{q, w} \abs{w}^{4 -
    \frac{2p}{q(p-2)}}_2 \abs{w}^{\frac{2p}{q(p-2)}}_p.
\end{equation}
For $p = 2^{\sharp}$ we get 
\begin{equation}
    \label{eq:stima_conv_sharp}
    \int_{\mathbb{R}^{N}} ( W* \abs{w}^2) \abs{w}^2 \, dy \leq C
    \norm{W}_{q, w} \abs{w}^{4 - \frac{2N}{q}}_2
    \abs{w}^{\frac{2N}{q}}_{2^{\sharp}}.
\end{equation}
In the (critical) case $q = N$ this gives
\begin{equation}
    \label{eq:stima_conv_critic}
    \int_{\mathbb{R}^{N}} ( W* \abs{w}^2) \abs{w}^2 \, dy \leq C
    \norm{W}_{N, w} \abs{w}^{2}_2 \abs{w}^{2}_{2^{\sharp}}.
\end{equation}
Let us point out that one cannot deduce \eqref{eq:stima_conv_critic}
from the weak Young's inequality \eqref{eq:weakYang} directly, and
that it is not true, in general, that $L^{\infty}$ norm of
$W*\abs{u}^{2}$ can be bounded by the $L^{2^{\sharp}}$ norm of $u$ if
$W \in L^{N}_{w}$.

For all $v \in H^{1}(\mathbb{R}_{+}^{N+1})$, we consider the
functional given by
\begin{multline*}
    \mathcal{I}(v) = \frac{1}{2} \bigl( \iint_{\mathbb{R}^{N+1}_{+}}
    (\abs{\nabla v}^{2} + m^{2} \abs{v}^{2}) \, dx\, dy -
    \int_{\mathbb{R}^{N}} m \abs{\tra{v}}^{2} \, dy \bigr) \\
    + \frac{\eta}{p} \int_{\mathbb{R}^{N}} \abs{\tra{v}}^{p} \, dy -
    \frac{\sigma}{4} \int_{\mathbb{R}^{N}} (W \ast
    \abs{\tra{v}}^{2})\abs{\tra{v}}^{2} \, dy.
\end{multline*}
In view of \eqref{eq:ultimaLabel} and \eqref{eq:stima_conv} all the
terms in the functional $\mathcal{I}$ are well defined if $p \in (2,
2^{\sharp}]$ and $W \in L^q_w(\mathbb{R}^{N})$ with $q \geq N$.

Remark that from \eqref{eq:disProdotto} with $p = 2$ follows that
\begin{equation}
    m\int_{\mathbb{R}^{N}} |\tra{v}|^{2} \, dy \leq
    2(m\norm{v}_{2})\norm{\nabla v}_{2} \leq
    \iint_{\mathbb{R}^{N+1}_{+}} (\abs{\nabla v}^{2} +
    m^{2}\abs{v}^{2}) \, dx\, dy
    \label{eq:positivita_quadratica}
\end{equation}
showing that the quadratic part in the functional $\mathcal{I}$ is
nonnegative.

Moreover the following property can be easily verified
\begin{lemma}
    \label{lem:betterminimizer}
    For $u \in H^{1}(\mathbb{R}^{N+1}_{+})$, let $w = \tra{u} \in
    H^{1/2}(\mathbb{R}^{N})$, $\hat{w} = \mathcal{F}(w)$ and
    \begin{equation*}
	v(x,y) = \mathcal{F}^{-1}(e^{-x\sqrt{m^{2}+\abs{\cdot}^{2}}}
	\hat{w}) = \int_{\mathbb{R}^{N}}
	e^{-x\sqrt{m^{2}+\abs{\xi}^{2}}} \hat{w}(\xi) e^{i\xi y} \,
	d\xi.
    \end{equation*}
    
    Then $v \in H^{1}(\mathbb{R}^{N+1}_{+})$, $\norm{v} =
    \norm{w}_{H^{1/2}}$ , $\mathcal{I}(v) \leq \mathcal{I}(u)$ and
    $\mathcal{I}(v) = \mathcal{E}[w]$.
\end{lemma}

\section{Minimization problem}

We consider the following minimization problem
\begin{equation}
    \label{eq:min_pb}
    I(M) = \inf \settc{ \mathcal{I}(v)}{v \in \mathcal{M}_{M} }
\end{equation}
where the manifold $\mathcal{M}_{M}$ is given by
\begin{equation*}
    \mathcal{M}_{M} = \settc{ v \in H^{1}(\mathbb{R}_{+}^{N+1})
    }{\int_{\mathbb{R}^{N}} \abs{\tra{v}}^{2} =M}
\end{equation*}

\begin{rema}
    The term $m\int_{\mathbb{R}^{N}} |\tra{v}|^{2}$ in the functional
    $\mathcal{I}(v)$ is constant for all $v \in \mathcal{M}_{M}$.  The
    presence of such a term will allow us to show that the infimum of
    the functional $\mathcal{I}$ on $ \mathcal{M}_{M}$ is negative.
\end{rema}

Concerning the existence of a minimizer for problem \eqref{eq:min_pb}
we start by proving, in the following lemmas, boundedness from below
of functional $\mathcal{I}$ on $\mathcal{M}_{M}$ and some properties
of the infimum $I(M)$.

\begin{lemma}
    \label{lem:coercitivita}
    The functional $\mathcal{I}$ is bounded from below and coercive on
    $\mathcal{M}_{M} \subset H^{1}(\mathbb{R}^{N+1}_{+})$ for all $M >
    0$ if $\eta > 0$ or $q > N$ and for all $M$ small enough if $\eta
    = 0$ and $q = N$.
\end{lemma}

\begin{proof}
    Let us examine first the convolution term. 
    
    If $\eta > 0$, from \eqref{eq:stima_conv} and
    $\abs{\tra{u}}_{2}^{2} = M$ we have
    \begin{multline}
	\label{eq:termineconvoluzione}
	0 \leq \int_{\mathbb{R}^{N}} (W*\abs{\tra{u}}^2)
	\abs{\tra{u}}^2 \leq C \norm{W}_{q, w} \abs{\tra{u}}^{4 -
	\frac{2p}{q(p-2)}}_2 \abs{\tra{u}}^{\frac{2p}{q(p-2)}}_p \\
	= C
	\norm{W}_{q, w} M^{2 - \frac{p}{q(p-2)}}
	\abs{\tra{u}}^{\frac{2p}{q(p-2)}}_p.
    \end{multline}
    
    Since $\frac{2p}{q(p-2)} < p$ by assumption, this is enough to
    prove coercivity if $\eta > 0$.  Indeed we have in such a case
    that
    \begin{equation*}
	\mathcal{I}(u) \geq \frac{1}{2}\norm{u}^{2} -\frac{1}{2} mM +
	C_{1}\abs{\tra{u}}_{p}^{p} - C_{2}
	\abs{\tra{u}}^{\frac{2p}{q(p-2)}}_p \geq \frac{1}{2}
	\norm{u}^{2} - C_{3}.
    \end{equation*}
    
    In case $\eta = 0$ we  deduce from \eqref{eq:stima_conv_critic} 
    and
    $\abs{\tra{u}}_{2^{\sharp}} \leq C\norm{u}$ that
    \begin{equation*}
	\mathcal{I}(u) \geq \norm{u}^{2} - mM - 
	C\norm{W}_{q,w}M^{2-N/q}\norm{u}^{2N/q}
    \end{equation*}
    It is then clear that the functional is bounded below and coercive
    whenever $q > N$ and, in case $q = N$, if $\norm{W}_{N,w}M$ is
    small enough.
\end{proof}

\begin{lemma}
    $I(M) < 0$ for all $M > 0$.
\end{lemma}

\begin{proof}
    Take any function $u \in C^{\infty}_{0}(\mathbb{R}^{N})$,
    $\abs{u}_{2}^{2} = M$, and let $w(x,y) = e^{-mx}u(y)$.
    
    Then 
    \begin{multline*}
	I(M) = \inf_{v \in \mathcal{M}_{M}} \mathcal{I}(v) \leq \mathcal{I}(w)\\
	= \frac{1}{2} \iint_{\mathbb{R}^{N+1}_{+}}
	\bigl(\abs{\partial_{x}w}^{2} + \abs{\nabla_{y}w}^{2} + m^{2}
	\abs{w}^{2} \bigr) \, dx \, dy - \frac{m}{2} \int_{\mathbb{R}^{N}} \abs{u}^{2} \, dy+
	G(u) \\
	= \frac{m}{4} \int_{\mathbb{R}^{N}} \abs{u}^{2}   \, dy+
	\frac{1}{4m} \int_{\mathbb{R}^{N}} \abs{\nabla_{y}u}^{2}  \, dy +
	\frac{m}{4} \int_{\mathbb{R}^{N}} \abs{u}^{2}  \, dy - \frac{m}{2}
	\int_{\mathbb{R}^{N}} \abs{u}^{2}  \, dy + G(u)\\
	= \frac{1}{4m} \int_{\mathbb{R}^{N}} \abs{\nabla_{y}u}^{2}  \, dy  +
	G(u)
    \end{multline*}
    where
    \begin{equation*}
        G(u) = \frac{\eta}{p} \int_{\mathbb{R}^{N}}  \abs{u}^{p} \, dy - 
	\frac{\sigma}{4} \int_{\mathbb{R}^{N}}  (W*\abs{u}^{2})\abs{u}^{2} \, dy
    \end{equation*}
    Take now, for $\lambda > 0$, $u_{\lambda}(y) =
    \lambda^{N/2}u(\lambda y)$ and $w_{\lambda}(x,y) =
    e^{-mx}u_{\lambda}(y) \in \mathcal{M}_{M}$ for all $\lambda > 0$.
    We find that
    \begin{multline*}
	I(M) \leq \inf_{\lambda > 0} \mathcal{I}(w_{\lambda})\\
	\leq \inf_{\lambda \in (0,1)} \biggl[\frac{1}{4m}
	\int_{\mathbb{R}^{N}} \abs{\nabla_{y}u_{\lambda}}^{2} +
	\frac{\eta}{p} \int_{\mathbb{R}^{N}} \abs{u_{\lambda}}^{p} -
	\frac{\sigma}{4} \int_{\mathbb{R}^{N}} (W*
	\abs{u_{\lambda}}^{2}) \abs{u_{\lambda}}^{2} \biggr] \\
	\leq \inf_{\lambda \in (0,1)} \biggl[\frac{\lambda^{2}}{4m}
	\int_{\mathbb{R}^{N}} \abs{\nabla_{y}u}^{2} +
	\frac{\eta\lambda^{N(\frac{p}{2} - 1)}}{p}
	\int_{\mathbb{R}^{N}} \abs{u}^{p} -
	\frac{\sigma\lambda^{\alpha}}{4} \int_{\mathbb{R}^{N}}
	(W* \abs{u}^{2}) \abs{u}^{2} \biggr]
    \end{multline*}
    and since $\alpha < N(\frac{p}{2} - 1) < 2$ we have that the
    infimum is negative.
\end{proof}

\begin{lemma}
\label{lem:inf-prop}
    For all $M > 0$ and $\beta \in (0,M)$ we have that $I(M) <
    I(M-\beta) + I(\beta)$.  Moreover $\frac{I(M)}{M}$ is a concave
    function of $M$ and hence $I(M)$ is a continuous function of $M$.
\end{lemma}

\begin{proof}
    The subadditivity is a consequence of the fact that for all $\theta
    > 1$
    \begin{equation}
	\label{eq:Isublineare}
	I(\theta M) < \theta I(M) \quad \text{ which implies } \quad 
	\frac{1}{\theta}I(M) < I(M/\theta).
    \end{equation}
    Indeed, taking $\theta_{1} = \frac{M}{\beta}$ and $\theta_{2} =
    \frac{M}{M-\beta}$ we have that
    \begin{equation*}
	I(M) = \frac{\beta}{M}I(M) + \frac{M-\beta}{M} I(M) < I(\beta)
	+ I(M-\beta)
    \end{equation*}
    To prove that \eqref{eq:Isublineare} holds, we remark that for all
    $v \in \mathcal{M}_{M}$ and $\lambda = \theta^{1/2} > 1$ we have,
    thanks to \eqref{eq:positivita_quadratica}
    \begin{multline*}
	\mathcal{I}(\lambda v) = \frac{\lambda^{2}}{2}
	\bigl[\iint_{\mathbb{R}^{N+1}_{+}} (\abs{\nabla v}^{2} + m^{2}
	\abs{v}^{2} ) \, dx \, dy -
	m\int_{\mathbb{R}^{N}}\abs{\tra{v}}^{2} \, dy \, \bigr] \\
	+ \frac{\eta\lambda^{p}}{p} \int_{\mathbb{R}^{N}}
	\abs{\tra{v}}^{p} \, dy - \frac{\sigma\lambda^{4}}{4}
	\int_{\mathbb{R}^{N}} (W*\abs{\tra{v}}^{2}) \abs{\tra{v}}^{2}
	\, dy \leq \lambda^{4}\mathcal{I}(v)
    \end{multline*}
    Hence, since $I(M) < 0$
    \begin{multline*}
	I(\theta M) = \inf_{\abs{\tra{v}}_{2}^{2} =\theta M}
	\mathcal{I}(v) = \inf_{\abs{\tra{v}}_{2}= M}
	\mathcal{I}(\theta^{1/2}v) \\
	\leq \theta^{2} \inf_{\abs{\tra{v}}_{2}= M} \mathcal{I}(v) =
	\theta^{2} I(M) < \theta I(M) < I(M)
    \end{multline*}
    
    To prove the concavity of $\frac{I(M)}{M}$, we remark that
    \begin{equation*}
	\frac{I(M)}{M} = \frac{1}{M} \inf_{u \in \mathcal{M}_{M}}
	\mathcal{I}(u) = \inf_{u \in \mathcal{M}_{1}}
	\frac{\mathcal{I}(\sqrt{M}u)}{M}.
    \end{equation*}
    We now show that, for all $u \in \mathcal{M}_{1}$, $M \mapsto
    \mathcal{I}(\sqrt{M}u)/M$ is a concave function of $M$.  This will
    immediately prove that $I(M)/M$ is a concave function.
    
    Since
    \begin{multline*}
	\frac{\mathcal{I}(\sqrt{M}v)}{M} = \frac{1}{2} \bigl(
	\iint_{\mathbb{R}^{N+1}_{+}} (\abs{\nabla v}^{2} + m^{2}v^{2})
	\, dx \, dy - \int_{\mathbb{R}^{N}} m \abs{\tra{v}}^{2} \, dy  \bigr)\\
	+ \frac{\eta M^{p/2-1}}{p} \int_{\mathbb{R}^{N}}
	\abs{\tra{v}}^{p}  \, dy  - \frac{\sigma M}{4} \int_{\mathbb{R}^{N}}
	(W \ast \abs{\tra{v}}^{2})\abs{\tra{v}}^{2} \, dy
    \end{multline*}

    It is then immediate to check that the second derivative with respect to the variable $M$ is
    negative for all $M > 0$ when $p/2 < 2$ and that the function is
    linear when $p=4$ (namely the critical exponent for $N=2$).
\end{proof}

We are now ready to prove existence of a minimizer for the functional $\mathcal{I}$ on $\mathcal{M}_{M}$.

\begin{prop}
\label{prop:minimizer}
    For every $M > 0$ there is a function $u \in
    H^{1}(\mathbb{R}^{N+1}_{+})$ such that
    \begin{equation*}
        \begin{cases}
            \mathcal{I}(u) = I(M) &  \\
            \int_{\mathbb{R}^{N}} \abs{\tra{u}}^{2} \, dy = M & 
        \end{cases}
    \end{equation*}
    i.e.~a minimizer for $  \mathcal{I}$ in $\mathcal{M}_{M}$.
\end{prop}

\begin{proof}
    Let $\{ u_{n} \} \subset \mathcal{M}_{M}$ be a minimizing sequence.  Follows
    from lemma \ref{lem:betterminimizer} that also the sequence
    \begin{equation*}
	v_{n}(x,y) =
	\mathcal{F}^{-1}(e^{-x\sqrt{m^{2}+\abs{\cdot}^{2}}}
	\mathcal{F}(\tra{u_{n}}))
    \end{equation*}
    is a minimizing one.  From lemma \ref{lem:coercitivita} we deduce
    that $v_{n}$ is bounded in $H^{1}(\mathbb{R}^{N+1}_{+})$ and that
    $w_{n} \equiv \tra{v_{n}} = \tra{u_{n}}$ is bounded in
    $H^{1/2}(\mathbb{R}^{N})$ and   $\int_{\mathbb{R}^{N}} \abs{w_n}^{2} \, dy = M$.
    
    We will now use the concentration-compactness method of P.L. 
    Lions \cite{Lions84}.
    
    Namely, one of the following cases must occur
    \begin{description}
        \item[vanishing] for all $R > 0$ 
        \begin{equation*}
            \lim_{n\to+\infty} \sup_{z \in \mathbb{R}^{N}} 
	    \int_{z+B_{R}} \abs{w_{n}}^2  \, dy = 0 ; 
        \end{equation*}
    
        \item[dichotomy] for a subsequence $\{n_{k}\}$
	\begin{equation*}
	    \lim_{R\to+\infty} \lim_{k\to+\infty} \sup_{z \in
	    \mathbb{R}^{N}} \int_{z+B_{R}} \abs{w_{n_{k}}}^2  \, dy = \alpha
	    \in (0,M) ; 
        \end{equation*}
    
        \item[compactness] for all $\epsilon > 0$ there is $R > 0$, a 
	sequence $\{y_{k}\}$ and a subsequence $\{w_{n_{k}}\}$ such that
	\begin{equation*}
	    \int_{y_{k}+B_{R}} \abs{w_{n_{k}}}^2  \, dy \geq M - \epsilon. 
	\end{equation*}
    \end{description}
    Following the usual strategy we will show that vanishing and 
    dichotomy cannot occur.
    
    \begin{lemma}
    \label{ref:vanishing}
	If vanishing occurs, then
	\begin{equation*}
	    \int_{\mathbb{R}^{N}} (W * \abs{w_{n}}^{2}) \abs{w_{n}}^{2} \, dy\to 0 .
	\end{equation*}
    \end{lemma}
    
    \begin{proof}[Proof of lemma \ref{ref:vanishing}]
	Take any $\delta > 0$ and $R > 0$.  Let define $W_{\delta} = W
	\mathbb{I}_{\{W \geq \delta\}} $ and
	\begin{equation*}
	    W^{R}_{\delta} (|y|) = (W_{\delta}(|y|) - R) ^{+}
	    \mathbb{I}_{\{|y| < R\}} + W_{\delta}(|y|) \mathbb{I}_{\{|y|
	    \geq R\}},
	\end{equation*}
	where $\mathbb{I}_A$ is the characteristic function of the set
	$A$ .  Then it easy to check that $W \in
	L^{q}_{w}(\mathbb{R}^{N})$ implies that $W_{\delta} \in L^{s}
	(\mathbb{R}^{N}) $ for any $s \in [1,q)$ and moreover that
	$|W_{\delta}^{R}|_{s} \to 0$ as $R \to + \infty$ for any
	$\delta >0$.  Let us define also $\Gamma_{\delta}^{R} =
	W_{\delta} - W^{R}_{\delta} $.  It is clear that
	\begin{equation*}
	    0 \leq  (W-W_{\delta})(|y|) \leq \delta, \qquad 0 \leq
	    \Gamma^{R}_{d} (|y|) \leq R \qquad \forall y \in \mathbb{R}^{N}
	\end{equation*}
	Then for any given $\delta >0$ and $R>0$ and for some $s \geq
	N/2$ (which implies that $2 < 4s/(2s-1) \leq 2N/(N-1)$) we get
	from the Young inequality (also taking into account that by Sobolev embedding the
	sequence $\{w_{n}\}$ is bounded in $L^{p}$ for $p \in [2,2N/(N-1)]$ )
	\begin{multline*}
	    \int_{\mathbb{R}^{N}} (W * \abs{w_{n}}^{2}) \abs{w_{n}}^{2} \leq
	    \int_{\mathbb{R}^{N}} ((W-W_{\delta})* \abs{w_{n}}^{2})
	   \abs{ w_{n}}^{2} + \int_{\mathbb{R}^{N}} (W^{R}_{\delta}*
	    \abs{w_{n}}^{2}) \abs{ w_{n}}^{2} \\ + \int_{\mathbb{R}^{N}}
	    (\Gamma^{R}_{\delta}* \abs{w_{n}}^{2}) \abs{w_{n}}^{2} \\
	    \leq \delta \abs {w_n}_2^{4} + |W^{R}_{\delta}|_s
	    |w_n|_{4s/(2s-1)}^4 + R \iint_{\mathbb{R}^{N} \times
	    \mathbb{R}^{N}} \abs{w_{n}(y)}^{2} \abs{w_{n}(z)}^{2}
	    \mathbb{I}_{|z-y| \leq R} \, dy \, dz\\
	    \leq \delta M^2 + C |W^{R}_{\delta}|_s + R M \sup_{z \in
	    \mathbb{R}^{N}} \int_{z+B_{R}} \abs{w_{n}}^2 \, dy.
	\end{multline*}
	Now, letting first $n \to + \infty$, then $R \to + \infty$ and
	finally $\delta \to 0^{+} $ we conclude the proof of the lemma.
    \end{proof}
    \begin{lemma}
	\label{ref:dichotomy}
	If dichotomy occurs, then for any $\alpha \in (0,M)$ we have
	\begin{equation*}
	    I(M)  \geq  I(\alpha) + I(M- \alpha).
	\end{equation*}
    \end{lemma}
    \begin{proof}[Proof of lemma \ref{ref:dichotomy}]
	If dichotomy occurs then there is a sequence $\{n_{k}\}
	\subset \mathbb{N}$ such that for any $\epsilon >0$ there
	exists $R >0$ and a sequence $\{ z_{k} \} \subset
	\mathbb{R}^{N}$ such that
	\begin{equation*}
	    \lim_{k \to +\infty} \int_{z_{k}+B_{R}} \abs{w_{n_{k}}}^2
	     \, dy \in (\alpha - \epsilon,\alpha+\epsilon).
	\end{equation*}
	Let define $\tilde{w}_{k} = w_{n_{k}}(\cdot + z_k)$ and
	\begin{equation*}
	    \tilde{u}_{k}(x,y) =
	    \mathcal{F}^{-1}(e^{-x\sqrt{m^{2}+\abs{\cdot}^{2}}}
	    \mathcal{F}(\tilde{w}_{k} ))
	\end{equation*}
	so that $\{ \tilde{u}_{k} \}$ is a minimizing sequence for
	$\mathcal{I}$ on $\mathcal{M}_{M}$ such that
	\begin{equation*}
	    \lim_{k \to +\infty} \int_{B_{R}}
	    \abs{\gamma(\tilde{u}_{k}) }^2  \, dy \in (\alpha - \epsilon,
	    \alpha + \epsilon).
	\end{equation*}

	Since $\{ \tilde{u}_{k} \} $ is a bounded sequence in
	$H^{1}(\mathbb{R}^{N+1}_{+})$ then $\tilde{u}_{k} \to u$
	weakly in $H^{1}(\mathbb{R}^{N+1}_{+})$ and $\tilde{w}_{k} =
	\tra{\tilde{u}_{k}} \to w = \tra{u}$ weakly in $H^{1/2}$ and
	strongly in $L^{p}_{loc}(\mathbb{R}^{N})$ for $p \in [2,
	2N/(N-1))$.  Hence for all $\epsilon > 0$ there is $R > 0$
	such that
	\begin{equation*}
	    \int_{B_{R}} \abs{\tra{u} }^{2} \, dy = \lim_{k \to + \infty}
	    \int_{B_{R}} \abs{\gamma(\tilde{u}_{k}) }^2  \, dy \in (\alpha -
	    \epsilon, \alpha + \epsilon).
	\end{equation*}
	and 
	\begin{equation*}
	    \int_{\mathbb{R}^{N}}  \abs{\tra{u} }^{2} \, dy = \lim_{R \to 
	    +\infty} \int_{B_{R}} \abs{\tra{u} }^{2}  \, dy = \alpha.
	\end{equation*}

	We set $v_{k} = \tilde{u}_{k} - u$ and $\beta_{k}=
	\int_{\mathbb{R}^{N}} \abs{\tra{v_{k}}}^{2} \, dy $, by weak
	convergence in $L^{2}$ of the sequence $\{ \tra{\tilde{u}_{k}} \} $
	we get $ \lim_{k \to + \infty} \beta_k = M- \alpha$.
   
	Now we claim that 
	\begin{equation*}
	    I(M) = \lim_{k \to + \infty} \mathcal{I}(\tilde{u}_{k}) =
	    \mathcal{I}(u) + \lim_{k \to + \infty} \mathcal{I}(v_{k})
	    \geq I(\alpha) + \lim_{k \to + \infty } I(\beta_{k})
	\end{equation*}
	and by the continuity of the function $I$, as stated in lemma
	\ref{lem:inf-prop},  the lemma follows.
   
	Now let us prove the claim. We will show that 
	\begin{equation*}
	    \lim_{k \to +\infty} (\mathcal{I}(\tilde{u}_{k}) - 
	    \mathcal{I}(v_{k})) \to \mathcal{I}(u)
	\end{equation*}
	 Indeed by weak convergence in $H^1(\mathbb{R}^{N+1}_{+})$ we
	 immediately get
	 \begin{align*}
	     &\lim_{k \to + \infty} \left(
	     \iint_{\mathbb{R}^{N+1}_{+}} | \nabla \tilde{u}_{k} |^{2}
	     - \iint_{\mathbb{R}^{N+1}_{+}} | \nabla v_{k} |^{2}
	     \right)= \iint_{\mathbb{R}^{N+1}_{+}} | \nabla u |^{2} \\
	     &\lim_{k \to + \infty} \left(
	     \iint_{\mathbb{R}^{N+1}_{+}} | \tilde{u}_{k} |^{2} -
	     \iint_{\mathbb{R}^{N+1}_{+}} | v_{k} |^{2} \right)=
	     \iint_{\mathbb{R}^{N+1}_{+}} | u |^{2}
	 \end{align*} 
	 and by the Brezis-Lieb lemma
	 \begin{equation*}   
	     \lim_{k \to + \infty} \left( \int_{\mathbb{R}^{N}}
	     |\gamma(\tilde{u}_{k}) |^{p} - \int_{\mathbb{R}^{N}} | \gamma(
	     v_{k}) |^{p} \right)= \int_{\mathbb{R}^{N}} | \gamma(u) |^{p}
	 \end{equation*}
	 for $2 \leq p \leq  2N/(N-1)$. 
   
	 Hence we have to investigate the last nonlinear term.  We 
	 will show in the Appendix A that 
	 \begin{equation*}
	     \lim_{k \to + \infty} \left(\int_{\mathbb{R}^{N}} (W *
	     \abs{ \tilde{w}_{k} }^{2}) \abs{\tilde{w}_{k}}^{2} -
	     \int_{\mathbb{R}^{N}} (W * \abs{\gamma(v_{k} )}^{2})
	     \abs{\gamma(v_{k} )}^{2} \right) = \int_{\mathbb{R}^{N}}
	     (W * \abs{w}^{2}) \abs{w} ^{2}.
	 \end{equation*} 
	 from which the claim follows. 
     \end{proof}
     
     Finally, since we have ruled out both vanishing and dichotomy,
     then we may conclude that indeed {\it compactness} occurs, namely
     that for all $\epsilon > 0$ there is $R > 0$, a sequence $\{ y_{k} \}$
     and a subsequence $\{ w_{n_{k}} \}$ such that
     \begin{equation*}
	 \int_{y_{k}+B_{R}} \abs{w_{n_{k}}}^2 \, dy \geq M - \epsilon.
     \end{equation*}
     So let us define as
     before  $\tilde{w}_{k} = w_{n_{k}}(\cdot + y_k)$ and 
     \begin{equation*}
	\tilde{u}_{k}(x,y) =
	\mathcal{F}^{-1}(e^{-x\sqrt{m^{2}+\abs{\cdot}^{2}}}
	\mathcal{F}(\tilde{w}_{k} )).
    \end{equation*}
  
    Then $\tilde{u}_{k}$ is a minimizing sequence for $\mathcal{I}$ on
    $\mathcal{M}_{M}$ such that
    \begin{equation*}
	\int_{B_{R}} \abs{\gamma(\tilde{u}_{k}) }^2 \geq M - \epsilon.
   \end{equation*}
   Since $\{ \tilde{u}_{k} \} $ is a bounded sequence in
   $H^{1}(\mathbb{R}^{N+1}_{+})$ then $\tilde{u}_{k} \to u$ weakly in
   $H^{1}(\mathbb{R}^{N+1}_{+})$ and $\tilde{w}_{k} =
   \tra{\tilde{u}_{k}} \to w = \tra{u}$ weakly in $H^{1/2}$ and
   strongly in $L^{p}_{loc}(\mathbb{R}^{N})$ for $p \in [2,
   2N/(N-1))$.  As in the  proof of lemma \ref{ref:dichotomy} we deduce that $ \int_{\mathbb{R}^{N}} \abs{\tra{u} }^{2} = M$.
   
   Moreover we claim that as $k \to + \infty$
   \begin{equation*}
       \int_{\mathbb{R}^{N}} (W * \abs{ \tilde{w}_{k}}^{2}) \abs{\tilde{w}_{k}}
       ^{2} \to \int_{\mathbb{R}^{N}} (W * w^{2}) w^{2} .
   \end{equation*}
   Indeed, by the weak Young inequality and by H\"older inequality we
   have
   \begin{multline*}
       \left| \int_{\mathbb{R}^{N}} (W * \tilde{w}_{k} ^{2})
       \tilde{w}_{k} ^{2} - \int_{\mathbb{R}^{N}} (W * w^{2}) w^{2}
       \right| \leq \int_{\mathbb{R}^{N}} (W * (\tilde{w}_{k} ^{2} +
       w^2)) |\tilde{w}_{k} ^{2} - w^{2}| \\
       \leq C \|W\|_{q, w} |\tilde{w}_{k} ^{2} + w^2|_s |\tilde{w}_{k}
       ^{2} - w^2|_s \leq C |\tilde{w}_{k} - w|_{2s} \to 0
   \end{multline*}
   since $2 < 2s = 4q/(2q -1) < 2N/(N-1) $.

   Hence finally by weakly lower semicontinuity of $H^1$ and $ L^p$
   norms (the positive terms of the functional $\mathcal{I}$) we may
   conclude that
   \begin{equation*}
       \mathcal{I}(u) \leq \liminf_{k \to + \infty}
       \mathcal{I}(\tilde{u}_{k} ) = I(M)
   \end{equation*}
   which implies the $u$ is a minimizer for $ \mathcal{I}$ in
   $\mathcal{M}_{M}$.
\end{proof}
    
Now we collect all the results obtained to conclude the proof of
Theorem \ref{thm:Theo}.

\begin{proof}[Proof of Theorem \ref{thm:Theo}]
    By proposition \ref{prop:minimizer} there exists a function $u \in
    H^{1}(\mathbb{R}^{N+1}_{+})$ which minimizes $\mathcal{I}$ in
    $\mathcal{M}_{M}$.  Therefore $u$ can always be assumed
    nonnegative and, by lemma \ref{lem:betterminimizer}, of the form
    \begin{equation*}
	u(x,y) = \mathcal{F}^{-1}(e^{-x\sqrt{m^{2}+\abs{\cdot}^{2}}}
	\mathcal{F}(w))
    \end{equation*}
    where $w = \gamma(u) \in H^{1/2}(\mathbb{R}^{N})$.  
    
    If $W$ is a nonincreasing radial function, then $w$ can be assumed
    to be a radial nonincreasing function.  Indeed let $w^{*}$ be the
    spherically symmetric decreasing rearrangement of $w$ and define
    \begin{equation*}
	u^{*}(x,y) =
	\mathcal{F}^{-1}(e^{-x\sqrt{m^{2}+\abs{\cdot}^{2}}}
	\mathcal{F}(w^{*})).
    \end{equation*}
    Then $\mathcal{I}(u^{*}) = \mathcal{E}[w^{*}]$ (also this follows
    from lemma \ref{lem:betterminimizer}).  We can then use the
    properties of the spherically symmetric decreasing rearrangement,
    namely
    \begin{itemize}
	
	\item[(i)] $w^{*}$ is a nonnegative, radial function;

	\item[(ii)] $w \in L^{p}(\mathbb{R}^{N})$ implies $w^{*} \in
	L^{p}(\mathbb{R}^{N})$ and $|w^{*}|_{p} = |w|_{p}$;

	\item[(iii)] \emph{symmetric decreasing rearrangement
	decreases kinetic energy} (Lemma 7.17 in \cite{LiebLoss}), 
	that is
	\begin{equation*}
	    \int_{\mathbb{R}^{N}} w^{*} ( \sqrt{- \Delta + m^{2}} -m)
	    w^{*} \, dy \leq \int_{\mathbb{R}^{N}} w ( \sqrt{- \Delta
	    + m^{2}} -m) w \, dy ;
	\end{equation*}

	\item[(iv)] \emph{Riesz's rearrangement inequality} (see
	Theorem 3.7 in \cite{LiebLoss})), namely 
	\begin{equation*}
	    \int_{\mathbb{R}^{N}} (W \ast |w^{*}|^2) |w^{*}|^2 \, dy
	    \geq \int_{\mathbb{R}^{N}} (W \ast |w|^2) |w|^2 \, dy
	\end{equation*}
	if $W(y)=W^{*}(|y|)$ (in particular if $W$ is radial and 
	nonincreasing)
    \end{itemize}
    to deduce that
    \begin{equation*}
	\mathcal{I}(u^{*}) =\mathcal{E}[w^{*}] \leq \mathcal{E}[w] =
	\mathcal{I}(u) = I(M).
    \end{equation*}
  
    Moreover, by the theory of Lagrange multipliers, any minimizer $u
    \in H^{1}(\mathbb{R}^{N+1}_{+})$ of the functional $\mathcal{I}$
    on $\mathcal{M}_{M}$ is such that
    \begin{multline}
	\label{eq:deboleNeumann}
	\iint_{R^{N+1}_{+}} (\nabla u \nabla w + m^{2}u w) \, dx\, dy
	- \int_{\mathbb{R}^{N}} m \tra{u}\tra{w} \, dy + \mu
	\int_{\mathbb{R}^{N}} \tra{u}\tra{w} \, dy \\
	+ \eta \int_{\mathbb{R}^{N}} \abs{\tra{u}}^{p-2}\tra{u}
	\tra{w} \, dy - \sigma \int_{\mathbb{R}^{N}} (W \ast
	\abs{\tra{u}}^{2}) \tra{u}\tra{w} \, dy = 0\\
    \end{multline}
    for all $w \in H^{1}(\mathbb{R}^{N+1}_{+})$, i.e.~$u$ is a weak
    solution of the following nonlinear Neumann boundary condition
    problem
    \begin{equation}
    \label{eq:half_plane_eq}
	\begin{cases}
	    -\Delta u + m^{2}u =0 & \text{in } \mathbb{R}^{N+1}_{+} \\
	    - \frac{\partial u}{\partial x} + \mu u = m u - \eta
	    \abs{u}^{p-2}u + \sigma (W*|u|^{2}) u & \text{on }
	    \mathbb{R}^{N} = \partial\mathbb{R}^{N+1}_{+}
	\end{cases}
    \end{equation}
    for some Lagrange multiplier $\mu \in \mathbb{R}$.  To prove that
    $\mu > 0$ we take $w = u$ in \ref{eq:deboleNeumann} to get
    \begin{align*}
	0 &= \iint_{\mathbb{R}^{N+1}_{+}} (\abs{\nabla u}^{2} + m^{2}
	\abs{u}^{2}) \, dx\, dy - \int_{\mathbb{R}^{N}} m
	\abs{\tra{u}}^{2} \, dy + \mu \int_{\mathbb{R}^{N}}
	\abs{\tra{u}}^{2} \, dy \\
	&\qquad + \eta \int_{\mathbb{R}^{N}} \abs{\tra{u}}^{p} \, dy -
	\sigma \int_{\mathbb{R}^{N}} (W \ast \abs{\tra{u}}^{2})
	\abs{\tra{u}}^{2} \, dy \\
	&= 2 \mathcal{I}(u) + \mu \int_{\mathbb{R}^{N}}
	\abs{\tra{u}}^{2} \, dy 
	+ \eta( 1- \frac{2}{p}) \int_{\mathbb{R}^{N}}
	\abs{\tra{u}}^{p} \, dy \\
	&\qquad - \frac{\sigma}{2}
	\int_{\mathbb{R}^{N}} (W \ast \abs{\tra{u}}^{2})
	\abs{\tra{u}}^{2} \, dy.
    \end{align*}
    Since $\mathcal{I}(u) <0$ we have in particular that
    \begin{equation*}
	\frac{\eta}{p} \int_{\mathbb{R}^{N}} \abs{\tra{u}}^{p} \, dy <
	\frac{\sigma}{4} \int_{\mathbb{R}^{N}} (W \ast
	\abs{\tra{u}}^{2}) \abs{\tra{u}}^{2} \, dy
    \end{equation*}
    and hence, since $p \leq 2N/(N-1) \leq 4$, for $N \geq 2$, we get
    \begin{multline*}
	\mu \int_{\mathbb{R}^{N}} \abs{\tra{u}}^{2} \, dy \\
	= - 2 \mathcal{I}(u) - \eta( 1- \frac{2}{p})
	\int_{\mathbb{R}^{N}} \abs{\tra{u}}^{p} + \frac{\sigma}{2}
	\int_{\mathbb{R}^{N}} (W \ast \abs{\tra{u}}^{2})
	\abs{\tra{u}}^{2} \, dy \\
	> \eta ( \frac{4}{p} - 1) \int_{\mathbb{R}^{N}}
	\abs{\tra{u}}^{p} \, dy \geq 0.
    \end{multline*}

    Finally the regularity, the strictly positivity and the
    exponential decay at infinity of the weak nonnegative solutions of
    \eqref{eq:half_plane_eq} follow straightforwards from Theorems
    3.14 and 5.1 in \cite{CZNolasco2011}.
\end{proof}

\section{Appendix A}

We prove that
\begin{multline*}
    \int_{\mathbb{R}^{N}}| (W * w \gamma(v_{k} )) w \gamma(v_{k} ) | +
    \int_{\mathbb{R}^{N}} |(W * \gamma(v_{k} )^{2}) w^{2} |+
    \int_{\mathbb{R}^{N}} |(W * w \gamma(v_{k} ) ) w^2 | + \\
    + \int_{\mathbb{R}^{N}} |(W * \gamma(v_{k} )^2 ) w \gamma(v_{k} )
    |\to 0 \qquad \text{as} \ k \to + \infty.
\end{multline*}
as claimed in the proof of lemma \ref{ref:dichotomy}.  Indeed we have
the following result.
	
\begin{lemma}
    For any $w \in H^{1/2}(\mathbb{R}^{N})$ and for sequences $
    \{f_{n}, g_{n}, h_{n} \}$ bounded in $H^{1/2}(\mathbb{R}^{N})$ and
    such that $f_n \to 0$ in $L^{2}_{loc}$ we have
    \begin{equation*}
	\int_{\mathbb{R}^{N}} (W * |f_{n} g_{n}| ) |w h_{n} | \to 0
	\qquad \text{as} \ n \to + \infty.
    \end{equation*}
\end{lemma}
	
\begin{proof}
    It is convenient to introduce for any given $\delta >0$ and $R>0$,
    $W_{\delta} = W \mathbb{I}_{W \geq \delta} $ and
    \begin{equation*}
	W^{R}_{\delta} (y) = (W_{\delta} - R) ^{+} \mathbb{I}_{|y| <
	R} + W_{\delta} \mathbb{I}_{|y| \geq R} .
    \end{equation*}
  
    Then for $W \in L^{q}_{w}(\mathbb{R}^{N})$ we have $W_{\delta} \in
    L^{p} (\mathbb{R}^{N}) $ for any $ p \in [1,q)$ and moreover that
    $|W_{\delta}^{R} |_{p} \to 0$ as $R \to + \infty$ for any $\delta
    >0$.  Let introduce again also $\Gamma_{\delta}^{R} = W_{\delta} -
    W^{R}_{\delta} $.  Note that $\text{supp} \, \Gamma_{\delta}^{R}
    \subset B_{R}$ and $0 \leq \Gamma_{\delta}^{R} \leq R$.

    From Young inequality (with $p = N/2$, $r = 2p/(2p-1)= N/(N-1)$),
    H\"older inequality and Sobolev embedding we have
    \begin{equation}
	\label{eq:split}
	\begin{split}
	    \int_{\mathbb{R}^{N}} & (W * \abs{f_{n}g_{n}}) 
	    \abs{wh_{n}} \\
	    &\leq \int_{\mathbb{R}^{N}} ((W-W_{\delta})*
	    \abs{f_{n}g_{n}}) \abs{wh_{n}} 
	    + \int_{\mathbb{R}^{N}} (W^{R}_{\delta}*
	    \abs{f_{n}g_{n}}) \abs{wh_{n}} \\
	    &\qquad\qquad + \int_{\mathbb{R}^{N}}
	    (\Gamma^{R}_{\delta}* \abs{f_{n}g_{n}}) \abs{wh_{n}} \\
	    &\leq \delta |f_{n} g_{n}|_{1} |w h_{n} |_{1} +
	    |W^{R}_{\delta} |_{N/2} |f_{n} g_{n}| _{r}|w h_{n}| _{r} +
	    \int_{\mathbb{R}^{N}} (\Gamma_{\delta}^{R} * |f_{n} g_{n}|
	    ) |w h_{n} | \\
	    &\leq C (\delta + |W^{R}_{\delta} |_{N/2} ) +
	    \int_{\mathbb{R}^{N}} (\Gamma_{\delta}^{R} * |f_{n} g_{n}|
	    ) |w h_{n} | .\\
	\end{split}
\end{equation}
    First of all we claim that
    \begin{equation*}
	\int_{\mathbb{R}^{N}} (\Gamma_{\delta}^{R} * |f_{n} g_{n}| )
	|w h_{n} | \to 0 \qquad \text{as} \ n \to + \infty.
    \end{equation*}

    Indeed, for any $\epsilon >0$ we fix $R_{1} >0$ such that
    $|\mathbb{I}_{\mathbb{R}^{N} \setminus B_{1}} w|_{2} < \epsilon$,
    where $B_{1}= B_{R_{1}}$.  We introduce also $R_2 = R_{1} + R$ and
    $B_{2} = B_{R_{2}}$ so that for any $y \in B_{1} $ and $z \in
    \mathbb{R}^{N} \setminus B_{2}$, we have $|z-y| \geq R$ and hence
    $\Gamma_{\delta}^{R}(z-y) = 0$.

    Now we estimate the term as follows
    \begin{multline*}
	\int_{\mathbb{R}^{N}} (\Gamma_{\delta}^{R} * |f_{n} g_{n}| )
	|w h_{n} | = \int_{B_1} (\Gamma_{\delta}^{R} *
	(\mathbb{I}_{B_2}|f_{n} g_{n}|) ) |w h_{n} | +
	\int_{\mathbb{R}^{N} \setminus B_1} (\Gamma_{\delta}^{R} *
	|f_{n} g_{n}| ) |w h_{n} | \\
	\leq R |\mathbb{I}_{B_{2}} f_{n} g_{n}|_{1}
	|\mathbb{I}_{B_{1}} w h_{n}|_{1} + |\Gamma_{\delta}^{R} *
	(f_{n} g_{n}) |_{\infty}|\mathbb{I}_{\mathbb{R}^{N} \setminus
	B_{1}} h_{n}|_{2}|\mathbb{I}_{\mathbb{R}^{N} \setminus B_{1}}
	w|_{2} \\
	\leq R |g_{n}|_{2} |h_{n}|_{2} ( |\mathbb{I}_{B_{2}} f_{n}
	|_{2} |w|_{2}+ R |f_{n} |_{2} |\mathbb{I}_{\mathbb{R}^{N}
	\setminus B_{1}} w|_{2} ) \\
	\leq C R ( |\mathbb{I}_{B_{2}} f_{n} |_{2} +
	|\mathbb{I}_{\mathbb{R}^{N} \setminus B_{1}} w|_{2} ) \\
    \end{multline*}
    and since $f_n \to 0$ as $n \to + \infty$ in $L^{2}(B_2)$ the
    claim is proved.
	
    Then we conclude the proof of the lemma sending first $n \to +
    \infty$ , then $R \to + \infty$ and finally $\delta \to 0$ in
    \eqref{eq:split}.
\end{proof}

\providecommand{\noopsort}[1]{} \providecommand{\cprime}{$'$}
  \providecommand{\scr}{} \def\ocirc#1{\ifmmode\setbox0=\hbox{$#1$}\dimen0=\ht0
  \advance\dimen0 by1pt\rlap{\hbox to\wd0{\hss\raise\dimen0
  \hbox{\hskip.2em$\scriptscriptstyle\circ$}\hss}}#1\else {\accent"17 #1}\fi}
\providecommand{\bysame}{\leavevmode\hbox to3em{\hrulefill}\thinspace}
\providecommand{\MR}{\relax\ifhmode\unskip\space\fi MR }
\providecommand{\MRhref}[2]{  \href{http://www.ams.org/mathscinet-getitem?mr=#1}{#2}
}
\providecommand{\href}[2]{#2}

\end{document}